\documentclass[a4paper, 12pt] {article}
\usepackage[cp1251]{inputenc}
\usepackage[english]{babel}
\usepackage{amsthm}
\usepackage{amsfonts}
\usepackage{amsmath}
\usepackage{mathtext}
\usepackage{cite}
\usepackage{url}
\newcommand{\la}{\lambda}
\newcommand{\ee}{\varepsilon}
\newtheorem{theorem}{Theorem}
\newtheorem{lemma}{Lemma}
\newtheorem{proposition}{Proposition}
\theoremstyle{definition}
\newtheorem{example}{Example}
\newtheorem{ip}{Inverse problem}
\newtheorem{alg}{Algorithm}

\begin{document}
\begin{center}
{\large\bf A uniqueness theorem on inverse spectral problems for the Sturm--Liouville differential operators on time scales}\\[0.4cm]
{\bf M.A.\,Kuznetsova\footnote{Department of Mathematics, Saratov State University, Astrakhanskaya 83, Saratov
410012, Russia, email: kuznetsovama@info.sgu.ru}} \\[0.4cm]
%https://orcid.org/0000-0003-1083-0799
\end{center}

\thispagestyle{empty}

{\bf Abstract.} In the paper, Sturm--Liouville differential operators on time scales consisting of a finite number of isolated points and segments are considered. Such operators unify differential and difference operators.
We obtain properties of their spectral characteristics including asymptotic formulae for eigenvalues and weight numbers. Uniqueness theorem is proved for recovering the operators from the spectral characteristics.



\begin{center}\bf 1. Introduction \end{center}
\medskip

Time scale theory unifies discrete and continuous calculus. It has important applications in natural sciences,
engineering, economics and in other fields;  for examples see \cite{[1],[2],Hilger90,finance,fishing}. Models of processes in these cases include differential equations on a time scale, i.e. closed subset of the real line. Various aspects of differential equations on time scales including boundary value problems were considered in  \cite{[1],[2],Hilger90,Agarval,fractals,[7],zhang,Rynne, ozkan2019, yurko2019}. %The majority of results light up  spectral properties of the Sturm--Liouville operator on time scale of various structures. 

In this paper, we study inverse spectral problems for the Sturm--Liouville operator on time scales.
 Such problems consist in recovering
operators from given spectral characteristics. For the classical Sturm--Liouville operators on an interval, inverse problems have been
studied fairly completely; the classical results can be found in \cite{[3], [4], [5]}. 
However, nowadays there are only few works on inverse problem
theory for differential operators on time scales  because the statement and the study of inverse spectral problems essentially depend on the structure of the time scale.
 In particular, in \cite{[7]} an Ambarzumian type theorem was obtained for Sturm--Liouville operators on time scales.

We consider bounded time scales $T$ consisting of $N < \infty$ segments and $M < \infty$ isolated points:
\begin{equation} \label{time scale}
T = \bigcup_{l=1}^{N+M} [a_l, b_l], \quad  a_{l-1} \le b_{l-1} < a_l \le b_l , \, l = \overline{2, N+M}, \quad a_l < b_l \text{ iff } l \in \{ l_k \}_{k=1}^N,\end{equation}
% while  $a_l < b_l$  if and only if $l \in \{ l_k \}_{k=1}^N,$
 where $l_k$ denotes the indice corresponding to the $k$-th segment. The case $N=1,$ $M=0$ corresponds to the classical Sturm--Liouville operator.
 
 If $T$ consists only of isolated points, i.e. $N=0,$ we have a difference operator. %, which is widely used in numerical analysis. 
 Inverse spectral problems for the difference operators were studied in \cite{atkinson, triangle, filomat, gus1, gus2} and other works. In \cite{atkinson} the coefficients of finite discrete Sturm--Liouville type bondary value problem are recovered from the spectrum and the set of normalization constants or from two spectra. The works \cite{gus1, gus2} are devoted to the discrete analogues of inverse scattering problems on semiaxis and the whole axis. In \cite{triangle} V.~A.~Yurko studied the so-called operators of triangular structure, which generalize the difference ones, and proved the uniqueness theorem and obtained the algorithm for recovery from the Weyl matrix. 
In \cite{filomat} the uniqueness theorem for the inverse problem from the eigenvalues and the weight numbers of Sturm--Liouville type difference operator on a finite set of integers is proved. Let us note that this result is the particular case of Theorem~\ref{t3} in the present paper. Moreover, some numerical methods for solving inverse problems for ordinary differential operators are based on their approximations by difference operators (see \cite{numrec} and references therein). 
% In particular, some numerical methods for solving inverse problems for ordinary differential operators include their approximations with finite differences and recovery of the coefficients of the obtained difference operator, see \cite{numrec} and references therein. 

The paper is organized as follows. The Sturm--Liouville
operator on the time scale $T$ is introduced in Section~2. We study the following its spectral characteristics: the spectra of two boundary value problems with one common boundary
condition, the weight numbers and the Weyl function. In Section~3, we establish their asymptotical behavior (Theorems \ref{spectrum parts}--\ref{t2}). In Section~4, we study three inverse problems of recovering the potential of the Sturm--Liouville operator from the given Weyl function, the two spectra or the spectrum along with the weight numbers. The uniqueness theorem for these inverse problems is proved, see Theorem~\ref{t3}. We also offer Algorithm~\ref{alg}, which allows one to recover the potential of the difference Sturm--Liouville operator (i.e. when $N=0$). 

\medskip
\begin{center}\bf 2. Sturm--Liouville operators on time scales \end{center}
\medskip

For convenience of the reader here we provide necessary notions of the time scale theory (see \cite{[1],[2]} for more details). Let $T$ be an arbitrary closed subset of
${\mathbb R},$ which we refer to as the time scale. We define the so-called jump functions $\sigma$ and $\sigma_-$ on $T$ in the following
way:
$$
\sigma(x)=\left\{\begin{array}{cl}\inf \{s\in T:\; s>x\}, & x\ne \max T,\\[2mm]
\max T, & x=\max T,
\end{array}\right.
\sigma_-(x)=\left\{\begin{array}{cl}\sup \{s\in T:\; s<x\}, & x\ne \min T,\\[2mm]
\min T, & x=\min T.
\end{array}\right.
$$
A point $x\in T$ is called {\it left-dense}, {\it left-isolated}, {\it right-dense} and {\it right-isolated}, if $\sigma_-(x)=x,$
$\sigma_-(x)<x,$ $\sigma(x)=x$ and $\sigma(x)>x,$ respectively. If $\sigma_-(x)<x<\sigma(x),$ then $x$ is called {\it isolated}; if
$\sigma_-(x)=x=\sigma(x),$ then $x$ is called {\it dense}.

Denote $T^0:=T\setminus\{\max T\},$ if $\max T$ is left-isolated, and $T^0:=T,$ otherwise. We also denote by $C(B)$ the class of 
functions continuous on the subset $B \subseteq T.$

A function $f$ on $T$ is called $\Delta$-{\it differentiable} at $t \in T^0,$ if for any $\ee>0$ there exists $\delta>0$ such that
$$
|f(\sigma(t))-f(s)-f^{\Delta}(t)(\sigma(t)-s)|\le\ee |\sigma(t)-s|
$$
for all $s\in (t-\delta, t+\delta)\cap T.$ The value $f^{\Delta}(t)$ is called the $\Delta$-{\it derivative} of the function $f$ at the point~$t.$ 

The following proposition gives conditions of $\Delta$-differentiability at points of different types.

\begin{proposition} 1)  If $f(t)$ is $\Delta$-differentiable at $t,$ then $f(t)$ is continuous in $t.$

2) Let $t \in T$ be a right-isolated point.  Then $f$ is $\Delta$-differentiable at $t,$ if and only if $f$ is continuous in $t.$ In this
case we have
$$
f^{\Delta}(t)=\frac{f(\sigma(t))-f(t)}{\sigma(t)-t}.
$$

3) Let $t\in T$ be a right-dense point. Then $f$ is $\Delta$-differentiable at $t,$ if and only if there exists the limit
$$
\lim_{s\to t,\; s \in T} \frac{f(t)-f(s)}{t-s}=:  f^{\Delta}(t).
$$
In particular, if $(t - \ee, t + \ee) \subset T$ for some $\ee > 0,$ then $f$ is $\Delta$-differentiable at $t,$ if and only if $f$ is
differentiable at $t.$ In this case the equality $f^{\Delta}(t)=f'(t)$ holds.
\end{proposition}

We also introduce derivatives of the higher order $n \ge 2.$ Let the $(n-1)$-th $\Delta$-derivative $f^{\Delta^{n-1}}$ of $f$ be defined on
$T^{0^{n-1}},$ where $a^n = \underbrace{a \ldots a}_{n}$ for any symbol $a.$ If $f^{\Delta^{n-1}},$ in turn, is $\Delta$-differentiable on
$T^{0^n}:= (T^{0^{n-1}})^0,$ then $f^{\Delta^n}:= (f^{\Delta^{n-1}})^{\Delta}$ is called the {\it $n$-th $\Delta$-derivative} of $f$ on
$T^{0^n}.$ For $n\ge1$ we also denote by $C^{n}(T)$ the class of functions $f$ for which there exists the $n$-th $\Delta$-derivative
$f^{\Delta^n}$ and $f^{\Delta^n} \in C(T^{0^n}).$ 
From now on, $f^{\Delta^\nu}(x_1, \ldots, x_n)$ denotes the $\nu$-th $\Delta$-derivative of the function $f(x_1, \ldots, x_n)$ with respect to the first argument, and $f^{(\nu)}(x_1, \ldots, x_n)$ denotes its classical $\nu$-th derivative with respect to the first argument.

A function $F(t)$ is called antiderivative of $f(t),$ if there exists $F^\Delta(t) = f(t)$ for all $t \in T^0.$ In \cite[Section 1.4]{[1]} it was established that any function from $C(T^0)$ has antiderivatives, which differ by constant.  %The definite $\Delta$-integral is given by the following formula:
For any $a, b \in T,$ the formula 
$$\int_a^b f(t) \, \Delta t := F(b) - F(a)$$
defines the definite $\Delta$-integral of a function $f(t)$ on $T \cap [a, b].$

Consider the following Sturm--Liouville equation on $T:$
\begin{equation}
\ell y:=-y^{\Delta\Delta}(x)+q(x)y(\sigma(x))=\lambda y(\sigma(x)), \quad x\in T^{0^2}. \label{1}
\end{equation}
Here $\lambda$ is the spectral parameter, and $q(x)\in C(T^{0^2})$ is a real-valued function. 
A function $y$ is called a solution of equation (\ref{1}), if $y\in C^2(T)$ and equality \eqref{1} is fulfilled. 

For definiteness, we restrict ourself to time scales $T$ of the form \eqref{time scale}. We consider $N > 0 \text{ or }M \ge 3,$ otherwise equation \eqref{1} degenerates. Let us additionaly assume that $q \in W^1_2[a_{l_k}, b_{l_k}],$ $k=\overline{1,N}.$ Note that the last condition is equivalent to the belongingness of $q$ to the corresponding Sobolev-type space on $T$ (see \cite{Rynne}).  

For $T$ of the form \eqref{time scale} some concepts of time scale theory can be clarified. In particular, if $a, b \in \bigcup_{k=1}^{N+M}\{a_k, b_k\},$ $a \le b,$ then by additivity of $\Delta$-integral and \cite[Theorem 1.79]{[1]}
\begin{equation}\label{integral}
\int_a^b f(t) \, \Delta t = \sum_{k \colon b_k \in [a, b)} f(b_k) (a_{k+1} - b_k) + \sum_{k \colon a \le a_k < b_k \le b} \int_{a_k}^{b_k} f(t) \, dt.
\end{equation}

If $y$ satisfies \eqref{1}, then
$\Delta$-derivative of the function $y$ on  $T$ can be represented as
\begin{equation}
y^{\Delta}(b_l)=\frac{y(a_{l+1})-y(b_l)}{a_{l+1}-b_l},\;l=\overline{1,N+M-1};
\quad y^{\Delta}(x)=y'(x),\;x\in [a_{l_k},b_{l_k}),\; k=\overline{1,N},     \label{2}
\end{equation}
where $y'$ denotes the classical derivative of $y$.
By virtue of (\ref{2}), equation (\ref{1}) is equivalent to the system of $N$ Sturm--Liouville equations on the intervals:
\begin{equation}
-y''(x_k)+q(x_k)y(x_k)=\la y(x_k),\quad x_k \in (a_{l_k},b_{l_k}),\;k=\overline{1,N},  \label{4}
\end{equation}
along with the relations
\begin{equation}
y^{\Delta\Delta}(b_l)=\frac{1}{a_{l+1}-b_l}\left(y^\Delta(a_{l+1})-y^\Delta(b_l)\right)
=(q(b_l)-\la)y(a_{l+1}), \quad l \in {\cal S} \label{relations},\end{equation}
where 
$${\cal S}:= \{ l \colon 1 \le l \le N+M-1-\mu_1\}, \quad \mu_1 := \delta(a_{N+M}, b_{N+M}), \quad   \delta(k, n) := \left\{ \begin{array}{cc} 1, & k=n,\\ 0, & k\ne n. \end{array}\right.$$
According to \eqref{2},  relations \eqref{relations} are equalent to the following jump conditions:
\begin{equation}
\label{jump conditions} \left.
\begin{array}{cc}
y(a_{l+1}) = \alpha^l_{11}(\la) y(b_{l}) + \alpha^l_{12}(\la) y^\Delta(b_{l}), & l = \overline{1, N + M -1}, \\[3mm]
y^\Delta(a_{l+1}) = \alpha^l_{21}(\la) y(b_{l}) + \alpha^l_{22}(\la) y^\Delta(b_{l}), & l \in {\cal S},
\end{array}\right\}
\end{equation}
where $$
\begin{array}{cc} \alpha^l_{11}(\la) = 1, & \alpha^l_{12}(\la) = a_{l+1}-b_l, \\[3mm]
\alpha^l_{21}(\la) =  (a_{l+1}-b_l)(q(b_l) - \la), & \alpha^l_{22}(\la) = 1 + (a_{l+1}-b_l)^2(q(b_l) - \la).
 \end{array}$$
 Thus, equation (\ref{1}) on $T$ is equivalent to the system of equations (\ref{4}) along with the jump conditions (\ref{jump conditions}).
We arrange the coefficients of the jump conditions into the matrices
$$ \alpha^l(\la) := \begin{pmatrix}  \alpha^l_{11}(\la) &  \alpha^l_{12}(\la) \\
 \alpha^l_{21}(\la) &  \alpha^l_{22}(\la)\end{pmatrix},\; l \in {\cal S}; \quad \alpha^{N+M-1}(\la) := \big(\alpha^{N+M-1}_{11}(\la),\alpha^{N+M-1}_{12}(\la)\big), \; N+M - 1 \notin {\cal S}.$$

 Denote by $L_j$ the boundary value problem for Eq. (\ref{1}) on $T$
with the boundary conditions $$y^{\Delta^j}(a_1)=y(b_{N+M})=0, \quad j=0,1.$$
Let $S(x,\la)$ and $C(x,\la)$ be solutions of Eq. (\ref{1}) on $T$
satisfying the initial conditions
\begin{equation} \label{initial conditions}
S^{\Delta}(a_1,\la)=C(a_1,\la)=1,\; S(a_1,\la)=C^{\Delta}(a_1,\la)=0.
\end{equation}
For each fixed $x,$ the functions $S(x,\la)$ and
$C(x,\la)$ are entire in $\la$ of order $1/2.$ We introduce the entire functions
$$
\Theta_0(\la):=S(b_{N+M},\la), \ \Theta_1(\la) := C(b_{N+M}, \la).
$$
For $j=0,1$ eigenvalues $\{\la_{nj}\}_{n\ge 1}$ of the boundary value problem
$L_j$ coincide with zeros of the entire function $\Theta_j(\la),$ which is called the characteristic function
for $L_j.$ We provide several examples of the characteristic functions for various time scales and $q \equiv 0.$ Let us agree that $\lambda=\rho^2.$ %where $\mathrm{Re}\,\rho>0$ or $\mathrm{Re}\,\rho=0$ and $\tau:=\mathrm{Im}\,\rho\ge 0.$
\begin{example} \label{e1} Let $T = \{ l \}_{l=0}^3,$ i.e. $N=0,$ $M=4.$ Then $q(t)$ is defined in $t=0$ and $t=1.$ Using \eqref{jump conditions}, one can compute the characteristic functions 
$$\Theta_0(\la) = \la^2 -4\la+3, \quad \Theta_1(\la) = \la^2 -3\la+1.$$
Thus, we have  $\la_{10} = 1,$  $\la_{20} = 3,$ $\la_{11} = \frac{3 - \sqrt5}{2},$ $\la_{21} = \frac{3 + \sqrt5}{2}.$ 
\end{example}

\begin{example}Consider $T = [0, 1],$  i.e. $N=1,$ $M=0.$ In this case \eqref{1} coincides with the classical Sturm-Liouville equation, and
$$\Theta_0(\la) = \frac{\sin \rho}{\rho}, \ \Theta_1(\la) = \cos \rho, \quad \la_{nj} = \pi \Big(n-\frac{j}{2}\Big), \ n \ge 1, \ j=0,1.$$
\end{example}

\begin{example} Consider $T = [0, 1] \bigcup [2, 3],$  i.e. $N=2,$ $M=0.$ Then
$$\Theta_0(\la) = \cos^2 \rho + \frac{2-\la}{\rho} \cos \rho \sin \rho - \sin^2 \rho, \quad \Theta_1(\la) = (\la-1) \sin^2 \rho + \cos^2 \rho - 2 \rho \sin \rho \cos \rho.$$
With the standard method involving Rouche's theorem it can be established that $$\{\la_{nj}\}_{n \ge 1} = \Big\{ (\pi n + o(1))^2 \Big\}_{n \ge 0} \bigcup \Big\{ \Big(\pi \big(n - \frac{1-j}{2}\big) + o(1)\Big)^2 \Big\}_{n \ge 1-j}, \quad j=0,1.$$ 
\end{example}

In these examples we can conclude that both spectra are finite if and only if $N=0.$ In the last case each one contains $M-2$ numbers. 
This observation remains true in the general case. In the next section we prove it and obtain the asymptotic formulae when $N>0.$

To obtain the other necessary properties of the eigenvalues, we introduce the notion of the Wronskian-type determinant $W(\varphi, \psi):=\varphi(t)\psi^\Delta(t) - \varphi^\Delta(t)\psi(t),$ where $\varphi(t)$ and $\psi(t)$ are solutions of equation \eqref{2}. By virtue of Theorem~3.13 in \cite{[1]}, we have  $W(\varphi, \psi) \equiv const$ on $T^0.$

\begin{proposition} 1. The sequences $\{\la_{n0}\}_{n\ge 1}$ and $\{\la_{n1}\}_{n\ge 1}$ have no common elements.

2. All zeros of $\Theta_j(\la),$ $j=0,1,$ are real and simple.   
\end{proposition}

\begin{proof}
1. It is obvious that $W(C, S)=1$  for all $t \in T^0$ and $\la \in \mathbb C.$ Suppose that $S(b_{N+M}, \la_0)=C(b_{N+M}, \la_0) = 0$ for some $\la_0.$ If $a_{N+M} < b_{N+M},$ then $\{C(t, \la_0),  S(t, \la_0)\},$ $t \in (a_{N+M}, b_{N+M}),$ is a fundamental system of solutions of the $N$-th equation \eqref{4} in $\la=\la_0,$ and we arrive at the contradiction. In the case when $a_{N+M} = b_{N+M},$ from \eqref{jump conditions} follows the linear dependence of the vectors $$(S(b_{N+M-1}, \la_0), C(b_{N+M-1}, \la_0))^T, \ (S^\Delta(b_{N+M-1}, \la_0), C^\Delta(b_{N+M-1}, \la_0))^T,$$ where $T$ is the transposition sign. The latter contradicts to $W(C, S) = 1.$

2. Consider the case $j=1.$ Let $\Theta_1(\la_0) = 0$ and $t \in T^{0^2}.$ From \eqref{1} for $y = C(t, \la)$ and for $y = C(t, \la_0)$ one can obtain
$$-C^{\Delta\Delta}(t, \la){C(\sigma(t), \la_0)} + {C^{\Delta\Delta}(t, \la_0)} C(\sigma(t), \la) = (\la - \la_0) C(\sigma(t), \la) {C(\sigma(t), \la_0)}.$$
The relation $(f(t) g(t))^\Delta = f^\Delta(t) g(\sigma(t)) + f(t) g^\Delta(t)$ yields that 
$$(-C^{\Delta}(t, \la) {C(t, \la_0)} + {C^{\Delta}(t, \la_0)} C(t, \la))^\Delta = (\la - \la_0) C(\sigma(t), \la){C(\sigma(t), \la_0)}.$$
 Denote $t_{r} := \max {T^0}.$ Note that  $t_r = b_{N+M}$ when  $b_{N+M}$ is left-dense and $t_r = b_{N+M-1}$ in the opposite case.
Integrating both sides of the previous relation and using the initial conditions \eqref{initial conditions}, we get
\begin{equation} (\la-\la_0) \int_{a_1}^{t_r} C(\sigma(t), \la) {C(\sigma(t), \la_0)} \, \Delta t = {C^{\Delta}(t_{r}, \la_0)} C(t_{r}, \la)-{C(t_{r}, \la_0)} C^{\Delta}(t_{r}, \la).
\label{int}\end{equation}
Due to real-valuedness of $q,$ since $\la_0$ is an eigenvalue, the number $\overline{\la_0}$ is also an eigenvalue with the eigenfunction $C(t, \overline{\la_0}) = \overline{C(t, \la_0)}.$  Substituting $\la=\overline{\la_0}$ into \eqref{int}, we have
\begin{equation*} -2 \, \mathrm{Im} \, \la_0 \int_{a_1}^{t_r} |C(\sigma(t), \la_0)|^2 \, \Delta t = {C^{\Delta}(t_{r}, \la_0)} \overline{C(t_{r}, \la_0)}-{C(t_{r}, \la_0)} \overline{C^{\Delta}(t_{r}, \la_0)}.
\end{equation*}
Then the following relation is obvious in the case $t_r = b_{N+M}:$
\begin{equation} -2 \, \mathrm{Im} \, \la_0 \int_{a_1}^{t_r} |C(\sigma(t), \la_0)|^2 \, \Delta t = 0. \label{Im}\end{equation}
The formula is also valid when $t_r \ne b_{N+M}$ since  $C^{\Delta}(t_{r}, \la_0) = \frac{C(b_{N+M}, \la_0) - C(t_{r}, \la_0)}{b_{N+M} - t_r}$ in this case.

From \eqref{integral} it follows that
\begin{equation}\int_{a_1}^{t_r} |C(\sigma(t), \la_0)|^2 \, \Delta t = \sum_{k\colon b_k < t_{r}} (a_{k+1} - b_k)|C(a_{k+1}, \la_0)|^2 + \sum_{k=1}^N \int_{a_{l_k}}^{b_{l_k}} |C(t, \la_0)|^2 \, dt > 0. \label{pos} \end{equation} 
Indeed, in the case when $a_1 < b_1$ the function $C(x, \la_0) $ is non-zero one on $[a_1, b_1]$ and $\int_{a_1}^{b_1} |C(t, \la_0)|^2 \, dt > 0;$ when $a_1 = b_1$ the relation $C(a_2, \la_0) = C(a_1, \la_0) + C^\Delta(a_1, \la_0)(a_2-a_1) = 1$ is fullfiled, and since $a_1=b_1 < t_{r},$ we have $$\sum_{k\colon b_k < t_{r}} (a_{k+1} - b_k)|C(a_{k+1}, \la_0)|^2 \ge a_2-a_1 > 0.$$

From \eqref{Im} and \eqref{pos} we conclude that $\mathrm{Im} \, \la_0 = 0.$
Further, from \eqref{int}, real-valuedness of $C(t, \la_0)$ and the equality $\Theta_1(\la_0)=0$  we get in both cases $a_{N+M} < b_{N+M}$ and $a_{N+M}=b_{N+M}$  that
$\int_{a_1}^{t_r} |C(\sigma(t), \la_0)|^2 \, \Delta t = {C^{\Delta}(t_{r}, \la_0)} \Theta_1'(\la_0).$ From the equality $W(C, S)=1$  it follows that 
\begin{equation}\int_{a_1}^{t_r} |C(\sigma(t), \la_0)|^2 \, \Delta t = -\frac{\Theta'_1(\la_0)}{\Theta_0(\la_0)}. \label{positive}
\end{equation} 
 Thus, \eqref{pos} yields the simplicity of $\la_0$ as the zero of $\Theta_1(\la).$ The case $j=0$ can be treated analogously. \end{proof}

 Let $\Phi(x,\lambda),$ $x \in T,$ be a solution of equation (\ref{1}) satisfying the boundary conditions
\begin{equation*}
\Phi^\Delta(a_1,\lambda)=1,\quad \Phi(b_{N+M},\lambda)=0.                               \label{8}
\end{equation*}

We call $M(\lambda):=\Phi(a_1,\lambda)$ the Weyl function, which generalizes the classical Weyl function. It is obvious that
\begin{equation}
\Phi(x,\lambda)=S(x,\lambda)+M(\lambda)C(x,\lambda),                              \label{9}
\end{equation}
\begin{equation}
M(\lambda)=-\frac{\Theta_0(\lambda)}{\Theta_1(\lambda)}.                              \label{10}
\end{equation}
Put
  $$ \alpha_{n} :=  \mathop{Res}_{\la = \la_{n1}} M(\la) = -\frac{\Theta_0(\lambda_{n1})}{\Theta'_1(\lambda_{n1})}, \quad  n \ge 1.$$
We call $\alpha_n$ weight numbers. The numbers $1/\alpha_n$ generalize the classical weight numbers for the Sturm--Liouville operator. From \eqref{pos} and \eqref{positive} for $\la_0 \in \{ \la_{n1}\}_{n \ge 1}$ it follows that $\alpha_n>0$ for all $n.$ 

The Weyl function $M(\lambda),$ the spectra $\{ \lambda_{nj}\}_{n \ge 1}, j=0,1,$ and the weight numbers $\{\alpha_n\}_{n \ge 1}$ are called spectral characteristics.
In the next section we establish their properties including asymptotic formulae.

\begin{center}\bf 3. Properties of the spectral characteristics \end{center}

Let us put $d_k := b_{l_k} - a_{l_k},$ $k=\overline{1,N},$ where $l_k$ are determined in \eqref{time scale}.   Without loss of generality, we assume that $l_{k} < l_{k+1},$ $k=\overline{1, N-1}.$ Denote also  $l_0 := 1,$ $l_{N+1} := N+M,$ $\mu_0 := \delta(a_1, b_1)$  and
 \begin{equation*} \beta^l(\la) =: \left\{\begin{array}{cc}\begin{pmatrix}  \beta^l_{11}(\la) &  \beta^l_{12}(\la) \\
 \beta^l_{21}(\la) &  \beta^l_{22}(\la)\end{pmatrix}, & l=\overline{1, l_N-1}, \\[4mm]
 \big(\beta^{l}_{11}(\la),\beta^{l}_{12}(\la)\big), &  l=\overline{l_N, l_{N+1}-\mu_1},\end{array}\right. \quad
\end{equation*}
where $\beta^l(\la)$ are determined for $k=\overline{1, N+\mu_1}$ and $s=\overline{1, l_k - l_{k-1}}$ as follows: 
 \begin{equation*}
 \beta^{l_k - s}(\la) :=  \alpha^{l_{k}-1}(\la) \ldots \alpha^{l_k-s}(\la); \quad \beta^{l_N}(\la) := (1, 0), \, l_N = N+M.
\end{equation*}
By virtue of \eqref{jump conditions}, we have
$$
\big(y(a_{l_k}), y^\Delta(a_{l_k})\big)^T =  \beta^{l_k - s}(\la) \big(y(b_{l_k-s}), y^\Delta(b_{l_k-s})\big)^T, \quad k=\overline{1, N}, \; s=\overline{1, l_k - l_{k-1}},
$$
$$ y(a_{l_{N+1}}) = \beta^{l_{N+1} - s}(\la) \big(y(b_{l_{N+1}-s}), y^\Delta(b_{l_{N+1}-s})\big)^T, \quad s=\overline{1, l_{N+1} - l_N}.$$
Further we establish asymptotic formulae for the elements of $\beta^l(\la).$

\begin{lemma} \label{l1}
For $k=\overline{1, N+\mu_1},$ $s=\overline{1, l_k - l_{k-1}}$ the following asymptotic formulae are fulfilled:
\begin{equation}\beta_{ij}^{l_k - s}(\la) =  a^{l_k - s}_{ij} (\la^{s-2+i} + b^{l_k - s}_{ij} \la^{s-3+i} + O(\la^{s-4+i})), \quad i = \overline{1, 2-\delta(k, N+1)}, \ j=1,2,\label{beta_asymp}
\end{equation}
where 
\begin{equation}a^{l_k - s}_{ij} = (-1)^{s-2+i}(a_{l_{k}-s+1} - b_{l_{k}-s})^{j-2} (a_{l_{k}} - b_{l_{k}-1})^{i-2} \prod_{l=l_k-s}^{l_{k} - 1} (a_{l+1} - b_l)^2,\label{a}\end{equation}
\begin{equation}b^{l_k - s}_{ij} = -
\left(\sum_{l=l_k-s+2-j}^{l_{k}-3+i} (a_{l+1} - b_l)^{-2} +  \sum_{l=l_k-s+1}^{l_{k}-1} (a_{l+1} - b_l)^{-1}(a_{l} - b_{l-1})^{-1} + \sum_{l=l_k-s}^{l_{k}-3+i} q(b_l)\right).\label{b}\end{equation}
\end{lemma}

\begin{proof} Fix any $k\in \overline{1, N+\mu_1}.$ In the case $s=1$ formulae \eqref{beta_asymp}--\eqref{b} are checked directly. Let \eqref{beta_asymp}--\eqref{b} be fulfilled for some $s=\mu \in [1, l_k - l_{k-1}).$ Then 
$$\beta_{ij}^{l_k-\mu-1}(\la) =  \beta_{i2}^{l_k-\mu}(\la) \alpha_{2j}^{l_k-\mu-1}(\la) + \beta_{i1}^{l_k-\mu}(\la) \alpha_{1j}^{l_k-\mu-1}(\la), \quad i=\overline{1,2-\delta(k, N+1)}, \, j=1,2.$$
Using this formulae and the induction assumption, we get \eqref{beta_asymp} for $s=\mu+1$ with the coefficients $a_{ij}^{l_k-\mu-1} = -(a_{l_k-\mu}-b_{l_k-\mu-1})^j a_{i2}^{l_k-\mu}$ and 
$$b_{ij}^{l_k-\mu-1} = -q(b_{l_k-\mu-1}) - \frac{a_{i1}^{l_k-\mu}}{(a_{l_k-\mu}-b_{l_k-\mu-1})a_{i2}^{l_k-\mu}}+ \frac{1-j}{(a_{l_k-\mu}-b_{l_k-\mu-1})^{2}} + b_{i2}^{l_k-\mu}.$$ %The step for $j=2$ is analogous. 
Then representations \eqref{a} and \eqref{b} in $s=\mu$ yield \eqref{a} and \eqref{b} in $s=\mu+1.$ 
Thus, by induction \eqref{beta_asymp}--\eqref{b} is proved for  $s=\overline{1, l_k-l_{k-1}}.$  
\end{proof}

Let us split $T$ into the union of the sets
$$
T_m:=\bigcup_{k=m}^{N+M}\,[a_k, b_k],\ T_{m,0}:=\bigcup_{k=1}^{m-1}\,[a_k, b_k], \quad m=\overline{1,N+M-\mu_1},
$$
and consider the solutions $S_{m}(x, \la), C_{m}(x, \la)$ of Sturm--Liouville equation \eqref{1} on $T_m$ satisfying the initial conditions
$$S_{m}(a_{m}, \la) = S^\Delta_{m}(a_{m}, \la) - 1 = C_{m}(a_{m}, \la) - 1 = C^\Delta_{m}(a_{m}, \la) = 0.$$
If $T_m \bigcap T^{0^2} = \emptyset,$ then the functions $S_{m}$ and $C_m$ are completely determined by these initial conditions.
%Due to constancy of Wronskian-type determinant  $W(\varphi, \psi)$ for two solutions $\varphi, \psi$ of one and the same Sturm--Liouville equation on $T_m$, every other solution can be expanded with respect to the system $\{ S_{m}, C_{m}\}.$

The functions $S_{l_k}(x+ a_{l_k}, \la), C_{l_k}(x+ a_{l_k}, \la),$ $x \in [0, d_k],$ $k=\overline{1, N},$ %are sin- and cos-type solutions of the equations \eqref{4}, they 
can be obtained as the solutions of the following integral equations:
\begin{equation}\left.
\begin{array}{c} \displaystyle
S_{l_k}(x + a_{l_k}, \la) = \frac{\sin \rho x}{\rho} + \int_0^x \frac{\sin \rho (x - t)}{\rho} S_{l_k}(t + a_{l_k}, \la) \, q_k(t)\, dt, \\[3mm] \displaystyle
C_{l_k}(x + a_{l_k}, \la) = \cos \rho x + \int_0^x \frac{\sin \rho (x - t)}{\rho} C_{l_k}(t + a_{l_k}, \la) \,q_k(t)\, dt,
\end{array}\right\} \label{Volterra}
\end{equation}
where  $q_k(x) := q(a_{l_k} + x),$ $x \in [0, d_k].$ 
Obviously, $q_k \in W^1_2[0, d_k].$ Substituting the standard asymptotic formulae for $S_{l_k}(x + a_{l_k}, \la)$ and $C_{l_k}(x + a_{l_k}, \la)$
(see \cite[Sect.1.1]{[5]}) into \eqref{Volterra}, for $x=d_k$ we obtain 
\begin{equation} \label{KuMA_S}
S_{l_k}(b_{l_k}, \la) = \frac{\sin \rho d_k}{\rho}\left[1+ \frac{\tilde A_{1k}}{\rho^2}\right] - \frac{\cos \rho d_k}{\rho^2} \omega_k - \frac{1}{4 \rho^3} \int_0^{d_k} q'_k(t) \sin \rho(2t-d_k) \, dt + \frac{O(e^{|\tau |d_k})}{\rho^4},
\end{equation}

\begin{equation} \label{KuMA_S'}
S'_{l_k}(b_{l_k}, \la) = \cos \rho d_k \left[1+ \frac{\tilde A_{2k}}{\rho^2}\right] + \frac{\sin \rho d_k}{\rho}\omega_k +  \frac{1}{4 \rho^2} \int_0^{d_k} q'_k(t) \cos \rho(2t-d_k) \, dt + \frac{O(e^{|\tau |d_k})}{\rho^3},
\end{equation}

\begin{equation} \label{KuMA_C}
C_{l_k}(b_{l_k}, \la) = \cos \rho d_k \left[1+ \frac{\tilde A_{3k}}{\rho^2}\right] + \frac{\sin \rho d_k}{\rho}\omega_k - \frac{1}{4 \rho^2} \int_0^{d_k} q'_k(t) \cos \rho(2t-d_k) \, dt + \frac{O(e^{|\tau |d_k})}{\rho^3},
\end{equation}

\begin{equation}
\label{KuMA_C'}
C'_{l_k}(b_{l_k}, \la)  = -\rho \sin \rho d_k  \left[1+ \frac{\tilde A_{4k}}{\rho^2}\right]  + \cos \rho d_k \, \omega_k  - \frac{1}{4 \rho} \int_0^{d_k} q'_k(t) \sin \rho(2t-d_k) \, dt + \frac{O(e^{|\tau |d_k})}{\rho^2}.
\end{equation}
Here %$$r_{ik}(\la) = O(e^{|\tau |d_k}), \ \tau := \mathrm{Im}\, \rho, \quad \left\{ \left|r_{ik}\Big(\frac{\pi n}{2d_k}+O(n^{-1})\Big)\right|\right\}_{n=1}^\infty \in l^2,$$ 
$$\tau := \mathrm{Im}\,\rho, \quad \omega_k := \frac12 \int_0^{d_k} q_k(t) \, dt, \quad \tilde A_{ik} := \frac{(-1)^{[(i-1)/2]}q_k(0)}{4}+ \frac{(-1)^{i-1}q_k(d_k)}{4} - \frac{\omega_k^2}{2},$$  $i=\overline{1,4},$ $k=\overline{1,N},$ and $[x]$ denotes the integer part of $x.$

Denote ${\cal D}_0^m(\la) := S_m(b_{N+M}, \la),$  ${\cal D}_1^m(\la) := C_m(b_{N+M}, \la),$ $m=\overline{1,N+M-\mu_1}.$ In particular, $\Theta_j(\la) = {\cal D}_j^{1}(\la),$ $j=0,1.$ 
 We also introduce the functions $\Phi_m(x,\lambda),$ $x \in T_m,$ which are solutions of equation (\ref{1}), $m=\overline{1, N+M-\mu_1},$ satisfying the boundary conditions
$\Phi_m^\Delta(a_m,\lambda)=1,\quad \Phi_m(b_{N+M},\lambda)=0. $
One can obtain the following formulae, which are analogues of \eqref{9}, \eqref{10}: 
\begin{equation}
\Phi_m(x,\lambda)=S_m(x,\lambda)+M_m(\lambda)C_m(x,\lambda),                              \label{m9}
\end{equation}
where
\begin{equation}
M_m(\lambda)=-\frac{{\cal D}_0^{m}(\la)}{{\cal D}_1^{m}(\la)}.                              \label{m10}
\end{equation}

\begin{lemma} \label{l2} For $j=0,1$ the following representations hold:
\begin{equation} \label{Dkjcase0}
{\cal D}_j^{l}(\la) = \beta_{1,2-j}^l(\la), \quad l=\overline{l_N+1, N+M-1}, \end{equation}

\begin{multline} \label{Dkjcase1}
{\cal D}_j^{l_k-s}(\la) = \rho^{\mu_1-1}\beta^{l_k-s}_{2,2-j}(\la) \prod_{i=k}^{N-1} \beta_{22}^{l_i}(\la) \beta^{l_N}_{1,1+\mu_1}(\la) \left(\prod_{l=k}^N g_l(\rho) + \frac{O(e^{|\tau|\gamma_k})}{\rho^3}\right), \\ s=\overline{1,l_k-l_{k-1}-1+\delta(k,1)\mu_0},\end{multline}

\begin{equation} \label{Dkjcase2}
{\cal D}_j^{l_k}(\la) = (-1)^{j(1-\delta_k)}\rho^{\mu_1+j-1}\prod_{i=k}^{N-1} \beta_{22}^{l_i}(\la) \beta^{l_N}_{1,1+\mu_1}(\la)  \left(\prod_{l=k+1}^N g_l(\rho) v_{kj}(\rho) + \frac{O(e^{|\tau|\gamma_k})}{\rho^3}\right),\end{equation}
where $k=\overline{1,N}.$ For these $k$ and $j=0,1$ we denoted $\gamma_k:=\sum_{l=k}^N d_l,$ $\delta_k := \delta(l_k, N+M),$
$$g_k(\rho)  := v_{k0}(\rho) + (-1)^{\delta_k}\frac{v_{k1}(\rho)}{\rho(a_{l_k} - b_{l_k - 1})}, \quad l_k > 1,$$

\begin{multline}v_{kj}(\rho) :=f_{kj}(\rho)\left(1 + \frac{A_{kj}}{\rho^2}\right) +  f_{k,1-j}(\rho) \frac{c_k (-1)^{j+\delta_k}}{\rho}+\frac{(-1)^{\delta_k}}{4\rho^2} \int_0^{d_k} f_{kj}((2t/d_k - 1)\rho) q'_k(t) \, dt,
\label{v_kj}\end{multline}
where
$$
 f_{k0}(x) := \left\{ \begin{array}{cc}
\sin d_kx, & \delta_k = 1, \\
\cos d_kx, & \delta_k = 0,
\end{array} \right. \quad
 f_{k1}(x) := \left\{ \begin{array}{cc}
\cos d_kx, &\delta_k = 1, \\
\sin d_kx, & \delta_k = 0,
\end{array} \right. 
$$
$c_k$ and $A_{kj}$ are some constants, which can be expressed from $q_k:$
$$c_k := \left\{ \begin{array}{cc}
\omega_k, & \delta_k = 1, \\
\omega_k + \frac{1}{a_{l_k+1} - b_{l_k}}, & \delta_k=0,
\end{array} \right. \quad
A_{kj} := \left\{\begin{array}{cc}
\tilde A_{k,2j+1}, &  \delta_k = 1, \\
\tilde A_{k,2j+2} - \frac{\omega_k}{a_{l_k+1}-b_{l_k}}, & \delta_k = 0.
\end{array} \right.$$ \end{lemma}

{\it Proof.} We will prove these formulae by induction. For $k=N+M-\mu_1$ the formulae \eqref{Dkjcase0} or \eqref{Dkjcase2} is fulfilled: \eqref{Dkjcase0} follows from the jump conditions \eqref{jump conditions} while \eqref{Dkjcase2} follows from \eqref{KuMA_S}, \eqref{KuMA_C}.
Let ${\cal D}_j^{m+1}(\la)$ be given by formulae \eqref{Dkjcase2} for some $l_k = m+1>0.$
We consider two possible cases. First, let $l_{i} = l_k - 1,$ $i=k-1 > 0.$
Using \eqref{jump conditions} we expand $Y_0 := S_{m}$ and $Y_1:=C_{m}$ with respect to the system $\{C_{m+1}, S_{m+1} \}$ on $T_{m+1}:$
\begin{multline} \label{basic Dkj}
{\cal D}_j^{m}(\la) =(\alpha^{m}_{11}(\la) Y_j(b_{m}, \la) +\alpha^{m}_{12}(\la)  Y'_j(b_{m}, \la)){\cal D}_1^{m+1}(\la) + \\[3mm] (\alpha^{m}_{21}(\la)  Y_j(b_{m}, \la) + \alpha^{m}_{22}(\la) Y'_j(b_{m}, \la)){\cal D}_0^{m+1}(\la), \quad j=0,1.
\end{multline}
From \eqref{KuMA_S}--\eqref{KuMA_C'} and the definition of $\beta^l(\la)$ it follows that
\begin{equation*} \label{v_{N-j}2}
\alpha^{m}_{21}(\la)  Y_j(b_{m}, \la) + \alpha^{m}_{22}(\la) Y'_j(b_{m}, \la) = (-1)^j \rho^j \beta^{m}_{22}(\la) \left(v_{k-1,j}(\rho) + \frac{O(e^{|\tau|d_{k-1}})}{\rho^3}\right),\end{equation*}
\begin{multline}  \label{v_{N-j}1}
\alpha^{m}_{11}(\la) Y_j(b_{m}, \la) +\alpha^{m}_{12}(\la)  Y'_j(b_{m}, \la) = 
(-1)^{j} \rho^j \beta^{m}_{12}(\la) \left(v_{k-1,j}(\rho) + \frac{O(e^{|\tau|d_{k-1}})}{\rho^2}\right)
\\ =(-1)^{j+1} \rho^j \beta^{m}_{22}(\la) \left(\frac{\la^{-1} v_{k-1,j}(\rho)}{a_{l_{k}}-b_{l_{k}-1}} + \frac{O(e^{|\tau|d_{k-1}})}{\rho^4}\right).\end{multline}
These relations with \eqref{basic Dkj} and the induction assumption \eqref{Dkjcase2} give formulae \eqref{Dkjcase2} for $k=i.$ 

Second, let $a_{m} = b_{m}.$ Expanding $S_{m}, C_{m}$ with respect to the system $\{ S_{m+1}, C_{m+1}\}$  on $T_{m+1},$ we get
\begin{equation} \label{system}
{\cal D}_j^{m}(\la) =\alpha^{m}_{1,2-j}(\la) {\cal D}_1^{m+1}(\la)+ \alpha^{m}_{2,2-j}(\la) {\cal D}_0^{m+1}(\la), \quad j=0,1.
\end{equation}
Bracing $\alpha^{m}_{2,2-j}(\la)$ with \eqref{beta_asymp} and applying the induction assumption, we prove \eqref{Dkjcase1} in $s=1.$ 

The other cases are operated with the same technique. $\hfill\Box$

\medskip
The previous lemma yields that
\begin{equation}\Theta_j(\la) = \left\{ \begin{array}{cc}
F_j(\la) + O\left(\exp(\gamma_1|\tau|) \la^{N+M-2 +j(1-\mu_0)/2-\mu_1/2} \right), & N> 0,\\[3mm] 
\beta_{1,2-j}^1(\la), & N=0,
\end{array}\quad j=0,1,\right. \label{char}\end{equation}
where
$$ F_j(\la) = (-1)^{j(1-\delta_1)(1-\mu_0)} \rho^{\mu_1+j(1-\mu_0)-1}\prod_{k=1-\mu_0}^{N-1} \beta^{l_k}_{2,2-j\delta(0,k)}(\la) \beta^{l_N}_{1, 1+\mu_1}(\la)
 \prod_{k=2}^{N} f_{k0}(\rho) f_{1,(1-\mu_0)j}(\rho).$$

By the standard method involving Rouche's theorem \cite{[5]}, from \eqref{char} the following structure of the spectra can be established.

\begin{theorem} \label{spectrum parts}
Each spectrum consists of $N+1$ parts:
$$\{ \la_{nj} \}_{n \ge 1} = \Lambda_j \bigcup \bigg( \bigcup_{k=1}^N \big\{ (\rho^{(k)}_{nj})^2 \big\}_{n \ge 1} \bigg), \quad j=0,1,
$$
where $\Lambda_j$ contains $N+M+j(1-\mu_0) \mathrm{sign}(N-1+\mu_1)-\mu_1-1$ elements and for the subsequences $\big\{ (\rho^{(k)}_{nj})^2 \big\}_{n \ge 1}$ the following asymptotic formulae are fulfilled: 
\begin{equation}\rho^{(k)}_{nj} = \pi \frac{n - \delta^{j \delta(1,k)(1-\mu_0)}_k}{d_k} + o(1), \; \; \delta_k^0 := \frac12 \delta(\delta_k, 0), \; \; \delta^1_k := \frac12 - \delta_k^0, \quad 1 \le k \le N.
\label{asymp_eig}
\end{equation}
\end{theorem}

The main parts of eigenvalues' roots in \eqref{asymp_eig} from different subsequences can occur arbitrarily close to each other, which causes the difficulty in the further refinement of these asymptotic formulae.
To overcome it, we make the following additional assumption: 
\begin{equation} \label{commensurability}
d_k = r x_k, \; x_k \in \mathbb{Q}, \quad k = \overline{1, N}, \text{ for some } r>0,\end{equation} which %is required only for subsequent Theorems~\ref{t1}--\ref{t2} and is not used anywhere else. Assumption \eqref{commensurability} 
means commensurability of the segments. Analogous commensurability assumptions appear also in other situations, e.g. for studying  spectral properties of differential operators on geometrical graphs (see, e.g., \cite{bond}).
Assumption \eqref{commensurability} is needed for Theorems \ref{t1}--\ref{t2} and is not used anywhere else.
This assumption yields that for any fixed  $s,k\in \overline{1,N}$ and $j, \nu \in \{0, 1\}$ for all $l,n \in \mathbb{N}$
 we have  the following alternatives:  $$\frac{d_k}{d_s} = \frac{l - \delta^{j}_k}{n - \delta^{\nu}_s}\text{  or  }\left|\frac{d_k}{d_s} - \frac{l - \delta^{j}_k}{n - \delta^{\nu}_s}\right| \ge (Cn)^{-1},$$ where and in the sequel $C$ denotes different sufficiently large constants. Then we have
\begin{equation}f_{kj}\left(\pi \frac{n - \delta^{\nu}_s}{d_s}\right) = 0 \text{ or } C >\left|f_{kj}\left(\pi \frac{n - \delta^{\nu}_s}{d_s}\right)\right|> C^{-1}.
\label{approx}\end{equation}

 Denote by $\eta_{kj}(\rho)$ the multiplicity of $\rho$  as a zero of the function 
 $$\prod_{l=k+1}^N f_{l0}(\rho) f_{kj}(\rho), \quad k=\overline{1, N}.$$ 
For briefness denote different sequences from $l^2$ by one and the same symbol $\{ \kappa_n\}_{n \ge 1}.$
%One and the same symbol $\{ \kappa_n\}_{n \ge 1}$ denotes different sequences from $l^2.$  
We also use $\{\kappa_n(z)\}_{n \ge 1}$ to designate different sequences of functions which are continuous in some circle $|z|\le R$ with  $$\Big\{\max_{|z| \le R}|\kappa_n(z)|\Big\}_{n \ge 1}\in l^2.$$
 
The following theorem refines formulae \eqref{asymp_eig}  under the additional condition \eqref{commensurability}.
\begin{theorem} \label{t1}
If \eqref{commensurability} is fulfilled, for the subsequences $\big\{ (\rho^{(k)}_{nj})^2 \big\}_{n \ge 1}$ we have
\begin{equation} \rho_{nj}^{(k)} = \frac{\pi( n-\delta^{j \delta(1, l_k)}_k)}{d_k}  + O\left(\frac{1}{n}\right), \quad k = \overline{1, N}, \ n \in \mathbb{N}.
\label{eig O(1/n)} \end{equation}
\end{theorem}
\begin{proof}
We plan to use the formulae of Lemma~\ref{l2}. In the case $N=1$ the computations are analogous to the classical case of the Sturm--Liouville equation on interval since the problem of close eigenvalues' roots does not arise. Therefore, we consider only the case $N>1.$

For definiteness we consider $$\rho^{(N)}_{nj} =  K_{nj} + z_{nj} := \pi \frac{n-\delta^0_N}{d_N} + z_{nj}, \; z_{nj} = o(1), \; n \to \infty. $$
We substitute $\rho = \rho^{(N)}_{nj}$ into ${\cal D}^l_{\nu}(\rho^2),$ $\nu=0, 1,$ $l=\overline{1, N+M-\mu_1}.$ The following formulae can be proved by induction:
\begin{multline} \label{pol1}
{\cal D}^{l_k}_{\nu}((\rho^{(N)}_{nj})^2) = (\rho^{(N)}_{nj})^{\mu_1+\nu-1}\beta^{l_N}_{1,1+\mu_1}((\rho^{(N)}_{nj})^2)\prod_{i=k}^{N-1} \beta_{22}^{l_i} \big((\rho^{(N)}_{nj})^2\big) \\
\times \Bigg(c^{k \nu}_{nj} z_{nj}^{\eta_{k\nu}(K_{nj})} + \sum_{l=0}^{\eta_{k\nu}(K_{nj})-1} z_{nj}^l O(n^{l-\eta_{k\nu}(K_{nj})})\Bigg),\end{multline}
\begin{multline} \label{pol2}
{\cal D}^{l_k-s}_{\nu}((\rho^{(N)}_{nj})^2) = (\rho^{(N)}_{nj})^{\mu_1-1}\beta^{l_N}_{1,1+\mu_1}((\rho^{(N)}_{nj})^2)\prod_{i=k}^{N-1} \beta_{22}^{l_i} \big((\rho^{(N)}_{nj})^2\big) \beta_{2, 2-j}^{l_k-s} \big((\rho^{(N)}_{nj})^2\big) \\
\times  \Bigg(c^{k0}_{nj} z_{nj}^{\eta_{k0}(K_{nj})} + \sum_{l=0}^{\eta_{k0}(K_{nj})-1} z_{nj}^l O(n^{l-\eta_{k0}(K_{nj})})\Bigg),\end{multline}
where  $|c^{k\nu}_{nj}| \ge C^{-\eta_{k\nu}(K_{nj})}$ for sufficiently large $n,$ $s=\overline{1,l_k-l_{k-1}-1+\delta(k,1)\mu_0}.$ Their proof is conducted according the scheme of the one of Lemma~\ref{l2}; for formula \eqref{pol1} one should consider two cases $f_{k\nu}(K_{nj}) = 0$ and $C>|f_{k\nu}(K_{nj})|>C^{-1}$ due to \eqref{approx}. In the first case $v_{k\nu}(\rho^{(N)}_{nj}) = (\pm d_{k} + o(1)) z_{nj} + O(n^{-1})$ and the degree $\eta_{k\nu}(K_{nj})=\eta_{k+1,0}(K_{nj}) +1,$ in the second case $C>|v_{k\nu}(\rho^{(N)}_{nj})|>C^{-1}$ and $\eta_{k\nu}(K_{nj}) = \eta_{k+1,0}(K_{nj}).$

We note that ${\cal D}_j^1((\rho_{nj}^{(N)})^2) = 0.$ Thus, from \eqref{pol1} for $\mu_0 = 0$ or from \eqref{pol2} for $\mu_0 = 1$ we have
 $\left|z_{nj}\right|^N \le C^N \sum_{l=0}^{N} |z_{nj}|^l O(n^{l-N}),$
 which yields
 $$|y_{nj}|^N \le C^{N+1} \sum_{l=0}^{N-1} |y_{nj}|^l$$
 for $y_{nj}=  n z_{nj}$ and sufficiently large $C.$ From the last inequality it follows that $y_{nj} = O(1).$ Indeed, if  $|y_{nj}| > 2,$ we can estimate
 $$|y_{nj}| \le C^{N+1} \sum_{l=0}^{N-1} |y_{nj}|^{l+1-N} \le 2C^{N+1},$$
 and, hence, $\{ y_{nj}\}_{n=1}^\infty$ is bounded.

 Thus, we proved \eqref{eig O(1/n)} for $k=N.$ The other formulae can be proved analogously.
\end{proof}
Asymptotic formulae \eqref{eig O(1/n)} can be refined as well. Namely, the following theorem holds.
%%%%%%%%%%%
\begin{theorem}\label{precise eigenvalues} 
Denote $$z_k :=\frac{1}{\pi} \Big(c_k + \sum_{l=max(1, l_k-1)}^{l_k-1} (a_{l+1} - b_{l})^{-1}\Big), \quad k=\overline{1, N},$$
where $c_k$ are defined in Lemma~\ref{l2}.
If \eqref{commensurability} holds, then 
\begin{equation}\rho_{nj}^{(k)} = \frac{\pi( n-\delta^{j \delta(1, l_k)}_k)}{d_k} + \frac{z_k }{n-\delta^{j \delta(1, l_k)}_k} + \frac{r^{(k)}_{nj}}{n}, \quad r^{(k)}_{nj}=o(1), \quad k = \overline{1, N}, \ n \in \mathbb{N}.
 \label{kappa/n^2}\end{equation}
Provided all ${z_k}/{d_k},$ $k=\overline{1, N},$ are distinct, we have $r^{(k)}_{nj} = \kappa_n/n.$
 \end{theorem}
 \begin{proof}
By the same reason as for the proof of Theorem~\ref{t1}, we consider only the case $N > 1.$
Any point in the vicinity of $\rho^{(N)}_{nj}$ can be represented in the form $$\rho^{(N)}_{nj}(z) = \pi \frac{n - \delta_N^0}{d_N} + \frac{z}{n- \delta_N^0} =: K_{nj} + \frac{z}{n- \delta_N^0}, \quad |z| \le C.$$
Substituting $\rho=\rho^{(N)}_{nj}(z)$ into the ${\cal D}^l_\nu(\rho^2),$ $\nu=0,1,$ $l=\overline{1, N+M-\mu_1},$ analogously to \eqref{pol1} and \eqref{pol2} one can prove that
\begin{multline}\label{basic for wn}
{\cal D}_\nu^{l_k}((\rho^{(N)}_{nj}(z))^2) = (-1)^{\nu(1-\delta_k)}\rho^{\mu_1+\nu-1}\prod_{i=k}^{N-1} \beta_{22}^{l_i}((\rho^{(N)}_{nj}(z))^2) \beta^{l_N}_{1,1+\mu_1}((\rho^{(N)}_{nj}(z))^2) \\ \times\left(\prod_{l=k+1}^N g_l(\rho^{(N)}_{nj}(z)) v_{k\nu}(\rho^{(N)}_{nj}(z)) + \frac{\kappa_n(z)}{K_{nj}^{\eta_{k\nu}(K_{nj}) + 1}}\right), \; k=\overline{1,N},\end{multline}
\begin{multline}\label{basic for wn0}
{\cal D}_\nu^{l_k-s}((\rho^{(N)}_{nj}(z))^2) = \rho^{\mu_1-1}\beta^{l_k-s}_{2,2-j}((\rho^{(N)}_{nj}(z))^2) \prod_{i=k}^{N-1} \beta_{22}^{l_i}((\rho^{(N)}_{nj}(z))^2) \beta^{l_N}_{1,1+\mu_1}((\rho^{(N)}_{nj}(z))^2) \\ \times \left(\prod_{l=k}^N g_l(\rho^{(N)}_{nj}(z)) + \frac{\kappa_n(z)}{K_{nj}^{\eta_{k0}(K_{nj}) + 1}}\right), \; k=\overline{1,N},\, s=\overline{1,l_k-l_{k-1}-1+\delta(k,1)\mu_0}.\end{multline}

For the proof it is sufficient to obtain the analogue of \eqref{v_{N-j}1} with $\kappa_n(z)$ instead of $O(e^{|\tau|d_{k}})$ in the case when $f_{k\nu}(K_{nj}) = 0$ using Lemma~\ref{l1}; the other computations are similar to the proof of Theorem~\ref{t1}. 

Denote 
$$I:=\left\{ \begin{array}{cc}\{ m \colon f_{m0}(K_{nj}) = 0, m > 1\} \bigcup \{1\}, &  f_{1,j(1-\mu_0)}(K_{nj}) = 0, \\
\{ m \colon f_{m0}(K_{nj}) = 0, \,m > 1\}, & f_{1,j(1-\mu_0)}(K_{nj}) \ne 0. \end{array}\right.$$
Consider $\mu_0=0$  and \eqref{basic for wn} for $k = 1$ (for $\mu_0=1$ one uses  \eqref{basic for wn0}).
Using the condition \eqref{approx} we obtain
 \begin{multline}\Theta_j((\rho^{(N)}_{nj}(z))^2) = {\cal D}^{1}_{j}((\rho^{(N)}_{nj}(z))^2) = (\rho^{(N)}_{nj}(z))^{\mu_1+j-1}\beta^N_{1,1+\mu_1}((\rho^{(N)}_{nj}(z))^2)\prod_{i=1}^{N-1} \beta_{22}^{l_i}\big((\rho^{(N)}_{nj}(z))^2\big) 
 \\ \times\left(C_{nj} \prod_{m \in I}  \left(\frac{d_m z}{n- \delta_N^0} - \frac{d_N z_m}{n- \delta_N^0}\right) + \frac{\kappa_n(z)}{(n- \delta_N^0)^{\eta_{1j}(K_{nj})+1}}\right),\label{Theta_eig} \end{multline}
 where $|C_{nj}| \ge C^{\eta_{1j}(K_{nj})-N}.$
 With Rouche's theorem we obtain that $\Theta_j((\rho^{(N)}_{nj}(z))^2)$ has $\eta_{1,j(1-\mu_0)}(K_{nj})$ zeros $\rho^{(N)}_{nj}(z) = \pi \frac{n - \delta_N^0}{d_N} + \frac{z}{n- \delta_N^0}$ with $z = d_N z_m/d_m + o(1),$ $m \in I.$ This  means that $\Theta_j(\rho^2)$ has the following $\eta_{1,j(1-\mu_0)}(K_{nj})$ zeros which are close to $K_{nj}:$ 
 $$\pi \frac{n- \delta_N^0}{d_N} + \frac{d_N z_m + o(1)}{d_m (n- \delta_N^0)} = \pi \frac{l - \delta_m^{j\delta(1, l_m)}}{d_m} + \frac{z_m + o(1)}{l - \delta_m^{j \delta(1, l_m)}}, \quad m \in I,$$ such that $d_N/d_m = (n- \delta_N^0)/(l - \delta_m^{j \delta(1, l_m)})$ for some $l \in \mathbb{N}.$ In particular, we have \eqref{kappa/n^2} when $m=N \in I.$

Using \eqref{Theta_eig} in the vicinity $|z - z_N| < \delta$ for a sufficiently small $\delta$ it is easy to prove \eqref{kappa/n^2}  for $k=N$ with $r^{(k)}_{nj} = \kappa_n/n.$

The other formulae can be proved analogously.\end{proof}

Let us obtain asymptotic formulae for the weight numbers. %Earlier we proved that they are positive. 
For them one can prove the analogues of Theorems~\ref{spectrum parts}--\ref{precise eigenvalues}. However, for briefness we provide only  formulae under the conditions of Theorem~\ref{precise eigenvalues}.

\begin{theorem} \label{t2}
The sequence $\{ \alpha_n \}_{n \ge 1}$ consists of $N+1$ parts:
\begin{equation}\label{wn parts}\{ \alpha_n \}_{n \ge 1} = A %\bigcup_{k=1}^{N+M+j(1-\mu_0)-\mu_1-1} \alpha_n 
\bigcup \bigg(\bigcup_{k=1}^N \big\{ \alpha^k_n \big\}_{n \ge 1}\bigg), \quad
 \alpha^k_n := \mathop{Res}_{\la = (\rho^{(k)}_{n1})^2} M(\la), \ A := \{\mathop{Res}_{\la = z} M(\la) \colon z \in \Lambda_1\}.\end{equation}
% where the set $A$ consists of $N+M+j(1-\mu_0)(1-\delta(1, N))-\mu_1-1$ elements.
If \eqref{commensurability} is fulfilled and all $z_k/d_k,$ $k=\overline{1, N},$ are distinct, the following asymptotic formulae hold:
  \begin{equation} \label{wn precise}
  \alpha^k_n =\left\{ \begin{array}{cc} \displaystyle\frac{2}{d_1} \Big(1  + \frac{\kappa_n}{n}\Big), & k=1, \, \mu_{0} = 0, \\[4mm]
 \displaystyle  \frac{\kappa_n}{n}, & \text{ all the other cases}.\end{array} \right.
  \end{equation} \end{theorem}

\begin{proof} First of all, we note that \eqref{wn parts} follows from Theorem~\ref{spectrum parts}.
%With the account of Theorem~\ref{spectrum parts} formula  is obvious. %\eqref{asymp_eig} it follows that the sequence $\{ \alpha_n \}_{n \ge 1}$ consists of $N+1$ parts:
 
 Let us prove \eqref{wn precise}.
Consider the case $k=1,$ $\mu_0 = 0,$ $N > 1.$ Then $\delta_1^1 = 0,$ $\delta_1 = 0,$ $z_1=c_1,$ $f_{11}(x) = \sin d_1x,$ $f_{10}(x) = \cos d_1x.$ Let $\delta > 0$ be  a sufficiently small number such that $2\delta < |z_l/d_l - z_1/d_1|,$ $l=\overline{2,N}.$  Denote $$\rho_n(z) := \pi \frac{n}{d_1} + \frac{z_1 + z}{n} =: K_{n1} + \frac{z_1 + z}{n}, \quad |z| \le \delta.$$ By Cauchy's residue theorem we obtain
\begin{equation}\alpha^1_n= \frac{1}{2 \pi i} \int_{|z| = \delta} \frac{2 \rho_n(z)}{n}M(\rho^2_n(z)) \, dz
\label{Cauchy residue}
\end{equation}
for sufficiently large $n.$
Further, for $|z| \le \delta$ by \eqref{basic for wn} we get
  $$M(\rho^2_n(z)) = 
\frac{\displaystyle \prod_{l=2}^N g_l(\rho_n(z)) v_{10}(\rho_n(z)) + \frac{\kappa_n(z)}{n^{\eta_{10}(K_{n1}) + 1}}}{\displaystyle \rho_n(z) \Bigg(\prod_{l=2}^N g_l(\rho_n(z)) v_{11}(\rho_n(z)) + \frac{\kappa_n(z)}{n^{\eta_{11}(K_{n1}) + 1}}\Bigg)}.$$

 Since $\left|\prod_{l=2}^N g_l(\rho_n(z)) \right| \ge C^{-1} n^{-\eta_{11}(K_{n1})+1}$ on $|z| = \delta$ and $|\eta_{10}(K_{n1}) - \eta_{11}(K_{n1})| \le 1$, we have
 $$M(\rho^2_n(z)) =   \frac{ v_{10}(\rho_n(z)) + \frac{\kappa_n(z)}{n}}{\rho_n(z) \left(v_{11}(\rho_n(z)) + \frac{\kappa_n(z)}{n^2}\right)}.$$
 Substituting Taylor's formulae of $\sin$ and $\cos$  into \eqref{v_kj}, we write
$$v_{11}(\rho_n(z)) = \frac{(-1)^n d_1}{n} \left(z+ \frac{\kappa_n(z)}{n}\right), \quad  v_{10}(\rho_n(z)) = (-1)^n \left(1+ \frac{\kappa_n(z)}{n}\right).$$
% where 
%\begin{multline*} P_n(z) = \frac{A_{11}z_1}{K^2_{n1}} - \frac{d_1^2 (z_1 + z)^3}{6n^2} - \frac{B_{11}}{K_{n1}^2 \pi} + \frac{d_1^2 (z_1+z)^2}{2 \pi n^2}z_1 + \\+\frac{z_1 c_1}{K_{n1}^2}+ \frac{(-1)^n d_1}{4 \pi^2 n} \int_0^{d_1} \sin \frac{\pi n}{d_1}(2t - d_1) \, q'_1(t) \, dt. 
%\end{multline*}
% It can be proved with Rouche's theorem that the function $z + K^{-2}_{n1}(A_{11}+c_1)z + P_n(z)$ starting from some $n$ has a single zero $y_n = O(n^{-1})$ in the circle $|z| < \delta.$
Using \eqref{Cauchy residue} and the subsequent formulae, we get
$$
\alpha^1_{n} = \frac{2}{d_1 2 \pi i}\int_{|z| = \delta} \frac{1+ \frac{\kappa_n(z)}{n}}{z + \frac{\kappa_n(z)}{n}} dz  \\
=\frac{2}{d_1 2 \pi i}\int_{|z| = \delta} \Big(\frac{1}{z} + \frac{\kappa_n(z)}{n}\Big) \, dz.
$$
%Derivating the denominator in the simple zero $y_n$ and using the formula 
%\begin{equation}(1 + x)^{-1} = 1 - x + O(x^2) \label{simple Taylor}\end{equation} in $x = \frac{A_{11}+c_1}{K^2_{n1}} + P'_n(y_n) = O(n^{-2}),$ 
From this equation we obtain \eqref{wn precise} for $k=1, \mu_0=0.$

  The other cases can be operated analogously.\end{proof}

%  We claim that all numbers in the sets $\{ c_m, z_{m+1}, \ldots, z_N\}$ are distinct, $m=\overline{1, N-1}.$ Then the analogs of the formulae of Theorems 1 and 2 can be obtained for zeros of functions ${\cal D}^{l_m}_{j}(\la),$ and residues of  $-{\cal D}^{l_m}_{0}(\la)/{\cal D}^{l_m}_{1}(\la).$

%{\bf Remark.} In the case $N=0$ it is necessary for $\{ \la_{nj} \}_{n \ge 1},$ $j=0,1,$ and $\{ \alpha_n \}_{n \ge 1}$ to be the sets of $M-2$ numbers.

Now let us study the asymptotical behavior of the functions $C_{l_k}(x,\lambda)$ and  $\Phi_{l_k}(x,\lambda)$ in the case $N>0,$ $k=\overline{1, N}.$ For our purposes it is sufficient to
consider $\rho\in\Omega_{\delta}:=\{z:\arg z \in[\delta,\pi-\delta]\}$ and $x\in(a_{l_k},b_{l_k}).$ From \eqref{KuMA_C}--\eqref{KuMA_C'} it follows that
\begin{equation}  \label{16S}
C_{l_k}^{(\nu)}(x+a_{l_k}, \lambda) = \frac{(-i\rho)^\nu}2\exp(-i \rho x)[1], \quad x\in(0,d_k], \quad \rho\in\Omega_\delta, \quad \nu=0,1.
\end{equation}
%\begin{equation}  \label{16}
%S_{l_k}^{(\nu)}(x+a_{l_k}, \la) = - \frac{(-i \rho)^{(\nu)}}{2 i \rho} \exp(-i \rho x)[1], \quad x\in(0,d_k], \quad \rho\in\Omega_\delta, \quad \nu=0,1.
%\end{equation}

%Further, substituting asymptotic formulae \eqref{Dkjcase2}, \eqref{16S} and \eqref{16} into \eqref{m9} and \eqref{m10} we get
Using the standard approach (see, for example, \cite{yurko2019}) one can prove the following formulae:
\begin{equation}\label{Phi}
\Phi_{l_k}^{(\nu)}(x+a_{l_k},\lambda) =(i\rho)^{\nu-1}\exp(i\rho x)[1], \quad x\in[0,d_k).
\end{equation}

\begin{center}{\bf 4. Inverse problems}\end{center}

Consider the following three inverse problems. %As the spectral data we consider $M(\la)$ or two spectra $\{ \la_{nj} \}_{n \ge 1},$ $j=0,1,$ or spectrum $\{ \la_{n1} \}_{n \ge 1}$ along with the weight numbers $\{ \alpha_n \}_{n \ge 1}.$

 \begin{ip} Given $M(\la)$, find $q(x)$ on $T^{0^2}.$ \end{ip}
 
\begin{ip} Given $\{\lambda_{nj}\}_{n\ge 1},\, j=0,1,$ find $q(x)$ on $T^{0^2}.$\end{ip}

\begin{ip} Given $\{\lambda_{n1}\}_{n\ge 1}, \{\alpha_{n}\}_{n\ge 1},$ find $q(x)$ on $T^{0^2}.$\end{ip}

First, we show that these inverse problems are equalent, i.e. their input data uniquely determine each other. Since $\Theta_0(\la)$ and $\Theta_1(\la)$ have no common zeros, $\{\lambda_{n0}\}_{n\ge 1}$ and $\{\lambda_{n1}\}_{n\ge 1}$ are determined as zeros and poles of the Weyl function. Conversely, 
Hadamard's factorization theorem gives
$$
\Theta_j(\lambda) = C_j p_j(\lambda), \quad p_j(\lambda) = \lambda^{s_j}\prod_{\lambda_{nj}\ne0} \Big(1 -
\frac{\lambda}{\lambda_{nj}}\Big), \quad j=0,1,
$$
where $C_j$ is a non-zero complex constants, while $s_j$ is the multiplicity of the zero eigenvalue in the spectrum $\{\lambda_{nj}\}_{n\ge
1}.$ 

By virtue of \eqref{char}, the following limits exist:
$$
\lim_{\lambda\to i\infty}\frac{\Theta_j(\lambda)}{F_j(\lambda)}=1, \quad j=0,1,
$$
and, hence,
$$
C_j = \lim_{\lambda\to i\infty}\frac{F_j(\lambda)}{p_j(\lambda)}.
$$
Thus, the characteristic functions $\Theta_j(\lambda)$ are uniquely determined by their zeros $\{\lambda_{nj}\}_{n\ge 1}.$ Taking into account formula
(\ref{10}) we conclude that two spectra uniquely determine the Weyl function as well.

Using Lemmas~\ref{l1} and~\ref{l2}, one can prove by technique analogous to \cite{[5]} that the weight numbers and the poles uniquely determine the Weyl function by the formula
$$M(\la) = \left\{\begin{array}{cc} \displaystyle -\mu_0 (a_{2-\delta(N+M,1)}-a_1) + \sum_{n=1}^\infty \frac{\alpha_n}{\la - \la_{n1}}, & \quad N > 0, \\[3mm]
\displaystyle -(a_2 - a_1) + \sum_{n=1}^{N-2} \frac{\alpha_n}{\la - \la_{n1}}, & \quad N=0.\end{array}\right.$$
Thus, given the input data of one inverse problem, we can recover them of any other one. Moreover, both characteristic functions are determined by specifying the input data of any Inverse problem 1--3.

Further, using the ideas of the method of spectral mappings \cite{[5]} we prove the uniqueness theorem for the solutions of the inverse problems. For this purpose together with the boundary value problem $L_0$ we consider a problem $\tilde L_0$ of the same form but with
another potential $\tilde q.$ In this section we agree that if a certain symbol $\gamma$ denotes an object related to $L_0,$ then this symbol with tilde
$\tilde\gamma$ will denote the analogous object related to $\tilde L_0.$

\begin{theorem} \label{t3} If  one of the following conditions is fulfilled, then $q=\tilde q$ on $T^{0^2}:$
\begin{enumerate}
\item $M(\lambda)=\tilde M(\lambda);$
\item $\{\lambda_{nj}\}_{n\ge 1}=\{\tilde\lambda_{nj}\}_{n\ge 1},\, j=0,1;$
\item $\{\lambda_{n1}\}_{n\ge 1}=\{\tilde\lambda_{n1}\}_{n\ge 1}$ and $\{\alpha_n \}_{n \ge 1} = \{\tilde \alpha_n \}_{n \ge 1}.$
\end{enumerate}
Thus, specification of the spectral data of any type
uniquely determines the potential $q.$ \end{theorem}

\begin{proof} I. At first, fix $m \in \overline{1, N+M}$ such that $[a_m, b_m] \subseteq T^{0^2}$ and suppose that ${\cal D}_j^m(\la) \equiv \tilde{\cal D}_j^m(\la),$ $j=0,1.$  It follows from (\ref{m10}) that $M_m(\la) \equiv \tilde{M}_m(\la).$ Let us prove that $q$ and $\tilde q$ coincide on $[a_m, b_m].$

First, consider the case $a_m < b_m.$ For $x\in(a_m, b_m)$ we define the functions
$$ P_j(x,\la) = (-1)^j (\Phi_m(x, \la)\tilde C_m^{(2-j)}(x, \la) - \tilde \Phi^{(2-j)}_m(x, \la)C_m(x, \la)), \quad j = 1, 2.
%P_1(x,\lambda)=\tilde\Phi'_m(x,\lambda)C_m(x,\lambda)-\Phi_m(x,\lambda)\tilde C'_m(x,\lambda),\quad P_2(x,\lambda)=\Phi_m(x,\lambda)\tilde
%C_m(x,\lambda)-\tilde\Phi_m(x,\lambda)C_m(x,\lambda).
$$
By virtue of the relation $C_m(x,\lambda)\Phi'_m(x,\lambda)-C'_m(x,\lambda)\Phi_m(x,\lambda)\equiv1,$ we have
\begin{equation}\label{P}
%P_1(x,\lambda)\tilde\Phi(x,\lambda)+P_2(x,\lambda)\tilde\Phi'(x,\lambda)=\Phi(x,\lambda), \quad
P_1(x,\lambda)\tilde C_m(x,\lambda)+P_2(x,\lambda)\tilde C'_m(x,\lambda)=C_m(x,\lambda).
\end{equation}
It also follows from \eqref{16S}, \eqref{Phi} that for each fixed $x\in(a_m,b_m)$
\begin{equation}\label{18}
P_1(x,\lambda)=1+O\Big(\frac1\rho\Big),\quad P_2(x,\lambda)=O\Big(\frac1{\rho^2}\Big),\quad \rho\to\infty,\quad \rho\in\Omega_{\delta}.
\end{equation}
On the other hand, using (\ref{m9}) and the coinsidence of the Weyl functions, we get
$$  P_j(x,\la) = (-1)^j (S_m(x, \la)\tilde C_m^{(2-j)}(x, \la) - \tilde S^{(2-j)}_m(x, \la)C_m(x, \la)), \quad j = 1, 2.
%P_1(x,\lambda)=C_m(x,\lambda)\tilde S'_m(x,\lambda)-\tilde C'_m(x,\lambda)S_m(x,\lambda),\; P_2(x,\lambda)=\tilde
%C_m(x,\lambda)S_m(x,\lambda)-C_m(x,\lambda)\tilde S_m(x,\lambda),
$$
and consequently, for each fixed $x\in(a_m,b_m),$ the functions $P_1(x,\lambda)$ and $P_2(x,\lambda)$ are entire in $\lambda$ of order $1/2.$
By the Phragmen--Lindel\"of theorem and Liouville's theorem, asymptotics (\ref{18}) imply $P_1(x,\lambda)\equiv 1$ and
$P_2(x,\lambda)\equiv 0,$ which along with (\ref{P}) give $C_m(x,\lambda)=\tilde C_m(x,\lambda)$ for $x\in(a_m,b_m)$ and, by continuity, for
$x\in[a_m,b_m].$ Then $q(x)=\tilde q(x)$ for $x\in[a_m,b_m].$
%%%%%

Now let $a_m = b_m.$ The relation $a_m \in T^{0^2}$ means $m < N+M$ and $m < N+M-1$ 
if $\mu_1 = 1.$
If we prove that ${\cal D}_j^m(\la),$ $j=0,1,$ uniquely determine $q(a_m),$ this will yield $q(a_m)=\tilde q(a_m).$ 
Solving system \eqref{system} with respect to ${\cal D}_0^{m+1}(\la)$ and ${\cal D}_1^{m+1}(\la),$ we get
\begin{equation}{\cal D}_0^{m+1}(\la) = \alpha_{11}^m(\la) {\cal D}_0^m(\la) - \alpha^m_{12}(\la) {\cal D}_1^m(\la) = {\cal D}_0^m(\la) - (a_{m+1} - b_m){\cal D}_1^m(\la),\label{comp D_0}\end{equation}
\begin{equation}{\cal D}_1^{m+1}(\la) = \alpha^m_{22}(\la) {\cal D}_1^m(\la) -  \alpha_{21}^m(\la) {\cal D}_0^m(\la).\label{comp D_1}\end{equation}
Thus, by \eqref{comp D_0} the function ${\cal D}_0^{m+1}(\la)$ can be computed.
Let us write the asymptotic formulae for ${\cal D}_0^{m}(\la)/{\cal D}_0^{m+1}(\la),$ $\rho \in \Omega_\delta,$ $\lambda \to \infty.$
There are two possible cases:

Case 1. $a_{m+1} < b_{m+1}.$ Then $l_k=m+1$ for some $k \in \overline{1, N}.$ Use \eqref{Dkjcase1} for $l_k - s = m$ and \eqref{Dkjcase2} for $l_k=m+1:$
\begin{equation*}\frac{{\cal D}_0^m(\la)}{{\cal D}_0^{m+1}(\la)} = 
\alpha^m_{22}(\la) \frac{g_{k}(\rho) + o(\exp(|\tau|d_k)\rho^{-2})}{v_{k0}(\rho) + o(\exp(|\tau|d_k)\rho^{-2})}.\end{equation*}
Dividing the numerator and the denominator on $f_{k0}(\rho)$ and using the estimate $|f_{k0}(\rho)| \ge C^{-1} e^{|\tau| d_k}$ for $\rho \in \Omega_\delta,$ we get by \eqref{v_kj} that
\begin{equation}\frac{{\cal D}_0^m(\la)}{{\cal D}_0^{m+1}(\la)} = 
\alpha^m_{22}(\la) \frac{1+P_1(\rho) + P_2(\rho)  + o(\rho^{-2})}{1+P_1(\rho) + o(\rho^{-2})}, \label{d frac}
\end{equation}
where
$$ P_1(\rho) := (-1)^{\delta_k} \frac{f_{k1}(\rho)}{\rho f_{k0}(\rho)} c_k + \frac{A_{k0}}{\rho^2} + \frac{(-1)^{\delta_k}}{4 \rho^2 f_{k0}(\rho)} \int_0^{d_k} f_{k0}\left(\frac{2t\rho}{d_k}-\rho\right) q'_k(t)\,dt = O\left(\frac{1}{\rho}\right),$$
$$P_2(\rho) :=  \frac{(-1)^{\delta_k}}{\rho(a_{m+1} - a_m)}\left(\frac{f_{k1}(\rho)}{f_{k0}(\rho)} - (-1)^{\delta_k}\frac{c_k}{\rho}\right)= O\left(\frac{1}{\rho}\right).$$
One can also prove that $(-1)^{\delta_k} \frac{f_{k1}(\rho)}{f_{k0}(\rho)} = i + o(\rho^{-1}),$ $\rho \in \Omega_\delta.$ Then from the equalities $$(1+P_1(\rho) + o(\rho^{-2}))^{-1} = 1-P_1(\rho)+P^2_1(\rho)+o(\rho^{-2})$$
and \eqref{d frac} we obtain
\begin{equation*}\frac{{\cal D}_0^m(\la)}{{\cal D}_0^{m+1}(\la)} =(a_{m+1} - a_m)^2 (q(a_m)-\la) - (a_{m+1} - a_m)\rho i + 1 + o(1).
\label{iii}\end{equation*} 
From this formula one can compute the quantity $q(a_m).$ 

%Thus, the values $c_k$ and $q(a_m)$ are determined uniquely.

Case 2. $a_{m+1} = b_{m+1}.$ Let $N_m$ be the number of the indices in $\{ l_s \}_{s=1}^N$ which are greater then $m.$  Then the functions  ${\cal D}_0^m(\la)$ and ${\cal D}_0^{m+1}(\la)$ are given by \eqref{Dkjcase0} or by \eqref{Dkjcase1} depending on whether $N_m$ is zero or not respectively. Then
$$\frac{{\cal D}_0^m(\la)}{{\cal D}_0^{m+1}(\la)} = \frac{\beta_{2-\delta(N_m, 0),2}^m(\la)}{\beta_{2-\delta(N_m, 0),2}^{m+1}(\la)} \left(1 + o\left(\frac{1}{\rho^2}\right)\right).$$

Applying the formulae of Lemma~\ref{l1}, we obtain
\begin{multline}\frac{{\cal D}_0^m(\la)}{{\cal D}_0^{m+1}(\la)} =(a_{m+1} - a_m)^2 \\ \times\left(-\la +\frac{1}{(a_{m+1} - a_m)^2}  + \frac{1}{(a_{m+1} - a_m)(a_{m+2} - a_{m+1})}+ q(a_m)\right) + o(1).\label{comp_q}\end{multline}
With this relation $q(a_m)$ can be computed.

II. Let us prove by induction that the spectral data uniquely determine the potential $q(x).$ From the assumption of the theorem we have ${\cal D}_j^1(\la) = \tilde{\cal D}_j^1(\la),$ $j=0,1,$ and, by part I, $q \equiv \tilde q$ on $[a_1, b_1].$

Let $m \in \overline{1,N+M-1}$ be such that $[a_{m+1}, b_{m+1}] \subset T^{0^2}$ and \begin{equation}{\cal D}_j^{m}(\la) \equiv \tilde{\cal D}_j^{m}(\la), \; j=0,1, \quad  q(x)= \tilde q(x), \; x\in T_{m+1,0}. \label{coincidence}
\end{equation} In the case $a_{m+1} = b_{m+1}$  we obtain ${\cal D}_j^{m+1}(\la) \equiv \tilde{\cal D}_j^{m+1}(\la),$ $j=0,1,$ from formulae \eqref{comp D_0} and \eqref{comp D_1}. In the case $a_{m+1}<b_{m+1}$ the functions ${\cal D}_0^{m+1}(\la)$ and ${\cal D}_1^{m+1}(\la)$ are solutions of non-degenerate systems \eqref{basic Dkj}. Moreover, by \eqref{coincidence} the functions $\tilde{\cal D}_0^{m+1}(\la)$ and $\tilde{\cal D}_1^{m+1}(\la)$ are solutions of the same system, which yields ${\cal D}_j^{m+1}(\la) \equiv \tilde{\cal D}_j^{m+1}(\la),$ $j=0,1.$ By virtue of part~I, we conclude that $q \equiv \tilde q$ on $[a_{m+1}, b_{m+1}],$ and the theorem is proved by induction.
\end{proof}

Developing the ideas of the method of spectral mapping \cite{[5]}, one can obtain the algorithm for the recovery of the potential. %The computations of Theorem~\ref{t3} give the following algorithm for recovering the potential in the case $N=0,$ i.e. the recovery of the difference Sturm--Liouville operator. 
Here we restrict ourself the case of $N=0$ (i.e. the case of the difference Sturm--Liouville operator) since it is sufficient to have the computations of Theorem~\ref{t3} for the algorithm.

\begin{alg} \label{alg} Let the functions ${\cal D}_j^1(\la)=\Theta_j(\la),$ $j=0,1,$ be given. To recover $q(a_l),$ $l=\overline{1, M-2},$ for $m=\overline{1, M-2}$ do the following:

1) Construct the function ${\cal D}_0^{m+1}(\la)$ using \eqref{comp D_0} and the known functions  ${\cal D}_0^{m}(\la),$  ${\cal D}_1^{m}(\la).$

2) Find $q(a_m)$ from the relation \eqref{comp_q}. 

3) If $m < M-2,$ construct the function ${\cal D}_1^{m+1}(\la)$ using \eqref{comp D_1} and the found value $q(a_m).$ \end{alg}

\begin{example} Let us consider the time scale $T$ and the characteristic functions $\Theta_j(\la),$ $j=0,1,$ from Example~\ref{e1} and apply Algorithm~\ref{alg}.
First, we compute
 $${\cal D}_0^2(\la) = \Theta_0(\la) - \Theta_1(\la) = 2-\la, \quad  \frac{\Theta_0(\la)}{{\cal D}_0^2(\la)} = -\la +2 + \frac{1}{\la-2}.$$ 
 By formula \eqref{comp_q}, $2 = 2 + q(0)$ and $q(0) = 0.$ With \eqref{comp D_1} we find ${\cal D}_1^2(\la) = (1 - \la) {\cal D}_1^{m}(\la) + \la {\cal D}_0^{m}(\la) = 1-\la.$
 
 Further, ${\cal D}_0^3(\la) = {\cal D}_0^2(\la) - {\cal D}_1^2(\la) = 1$ and ${{\cal D}_0^2(\la)}/{{\cal D}_0^3(\la)} = 2 - \la.$  Applying \eqref{comp_q}, we conclude that $2 = 2 + q(1)$ and $q(1) = 0.$ Thus, the potential $q$ is recovered.
\end{example}

{\bf Acknowledgment.} This is a pre-print of an article published in Results in Mathematics. The final authenticated version is available online at: \\ \url{https://doi.org/10.1007/s00025-020-1171-z}.

This work was supported by Grant 19-71-00009 of the Russian Science Foundation.\\


\begin{thebibliography}{99}
%
\bibitem{[1]} M. Bohner, A. Peterson, Dynamic Equations on Time Scales, Birkh\"auser, Boston, MA, 2001.

\bibitem{[2]} M. Bohner, A. Peterson, Advances in Dynamic Equations on Time Scales, Birkh\"auser, Boston, MA, 2003.

\bibitem{Hilger90} S. Hilger,  Analysis on measure chains --- a unified approach to continuous and discrete calculus, Results in Math. 18 (1990), 18--56. 

\bibitem{[7]} S. Ozkan, Ambarzumian-type theorem on a time scale, Journal of Inverse and Ill-Posed Problems (2018), 1--5.

\bibitem{Rynne} B. P. Rynne, L2 spaces and boundary value problems on time-scales, J. Math. Anal. Appl. 328 (2007), 1217--1236.

\bibitem{Agarval} R. P. Agarwal , M. Bohner, D. O'Regan, Time scale boundary value problems on infinite intervals, Journal of Computational and Applied Mathematics 141 (2002), 27--34.

\bibitem{fractals} P. Amster, P. De N\'apoli,  J. P. Pinasco, Eigenvalue distribution of second-order dynamic equations on time scales considered as fractals, J. Math. Anal. Appl. 343 (2008), 573--584.
     
\bibitem{zhang} Y. Zhang, L. Ma, Solvability of Sturm--Liouville problems on time scales at resonance, Journal of Computational and Applied Mathematics 233 (2010), 1785--1797.

\bibitem{ozkan2019} A. S. Ozkan, I. Adalar, Half-inverse Sturm--Liouville problem on a time scale, Inverse Problems (2019). https://doi.org/10.1088/1361-6420/ab2a21

\bibitem{yurko2019} V. Yurko, Inverse problems for Sturm-Liouville differential operators on closed sets, Tamkang journal of mathematics 50:3 (2019), 199--206.

%\bibitem{Bohner18} M. Bohner,  S. Cebesoy, Spectral analysis of an impulsive quantum difference operator, Math. Meth. Appl. Sci. (2018), 1--9. %https://doi.org/10.1002/mma.5348
 
 \bibitem{atkinson} F. Atkinson, Discrete and Continuous Boundary Problems, Academic Press, New York, 1964.
 
 \bibitem{gus1} G. S. Guseinov, Determination of an infinite non-self-adjoint Jacobi matrix from
its generalized spectral function, Mat. Zametki 23:2 (1978), 237--248.

\bibitem{gus2}  G. S. Guseinov,  H. Tuncay, On the inverse scattering problem for a discrete one-dimensional Schr\"odinger equation, Comm. Fac. Sci. Univ. Ankara Ser. A1 44 (1995), 95--102.
 
 \bibitem{triangle} V. A. Yurko, An inverse problem for operators of a triangular structure, Results in Math. 30 (1996), 346--373. 
 
\bibitem{filomat} M. Bohner, H. Kemalo\u{g}lu (Koyunbakan), Inverse problems for Sturm--Liouville difference equations, Filomat 30 (2016), 1297--1304. %10.2298/FIL1605297B. 
 
 \bibitem{numrec} M. Ignatiev, V. Yurko, Numerical methods for solving Inverse Sturm--Liouville problems,  Results in Math. 52 (2008), 63--74.
 
%\bibitem{amb}  V. A. Ambarzumyan, \"Uber eine Frage der Eigenwerttheorie, Z. Phys. 53 (1929), 690--695.
 
\bibitem{[3]} V. A. Marchenko, Sturm--Liouville Operators and Their Applications, Naukova Dumka, Kiev, 1977; English transl., Birkh\"auser, 1986.

\bibitem{[4]} B. M. Levitan, Inverse Sturm--Liouville Problems, Nauka, Moscow, 1984; Engl. transl., VNU Sci. Press, Utrecht, 1987.

\bibitem{[5]} G. Freiling, V. A. Yurko,  Inverse Sturm--Liouville Problems and Their Applications, NOVA Science Publishers, New York, 2001.

\bibitem{bond} N. P. Bondarenko, An inverse problem for Sturm--Liouville operators on trees with partial information given on the potentials, Math. Meth. Appl. Sci. 42 (2019), 1512--1528. %https://doi.org/10.1002/mma.5448
%
\bibitem{finance} F. M. Atici, D. C. Biles, A. Lebedinsky, An application of time scales to economics, Mathematical and Computer Modelling 43 (2006), 718--726.

\bibitem{fishing} K. Prasad, K. Md, Stability of positive almost periodic solutions for a fishing model with multiple time varying variable delays on time scales, Bulletin of International Mathematical Virtual Institute 9 (2019), 521--533. %10.7251/BIMVI1903521P. 
\end{thebibliography}
\end{document}